\documentclass[12pt,a4paper]{article}

\usepackage[T1]{fontenc}
\usepackage[utf8]{inputenc}
\usepackage[british]{babel}
\usepackage{amsmath,amssymb,amsthm}
\usepackage{tensor}
\usepackage{xcolor}
\usepackage[all]{xy}
\usepackage{cite}
\usepackage{enumerate}
\usepackage{hyperref}

\newtheorem{theorem}{Theorem}

\newtheorem{corollary}[theorem]{Corollary}
\newtheorem{athm}{Theorem}

\theoremstyle{definition}

\newtheorem{question}[theorem]{Question}
\theoremstyle{remark}

\DeclareMathOperator\Aut{Aut}
\DeclareMathOperator\Cent{C}
\DeclareMathOperator\centre{Z}
\DeclareMathOperator\GF{GF}
\DeclareMathOperator\GL{GL}
\DeclareMathOperator\Hol{Hol}

\DeclareMathOperator\Inn{Inn}

\DeclareMathOperator\PSL{PSL}

\DeclareMathOperator\SL{SL}

\DeclareMathOperator\Ker{Ker}

\title{Constructing skew left braces whose additive group has trivial centre}
\author{A. Ballester-Bolinches\thanks{Departament de Matem\`atiques, Universitat de Val\`encia, Dr.\ Moliner, 50, 46100 Burjassot, Val\`encia, Spain; \texttt{Adolfo.Ballester@uv.es}, \texttt{Ramon.Esteban@uv.es}, \texttt{Vicent.Perez-Calabuig@uv.es}; ORCID 0000-0002-2051-9075, 0000-0002-2321-8139, 0000-0003-4101-8656}\and R. Esteban-Romero\addtocounter{footnote}{-1}\footnotemark\and P. Jim{\'e}nez-Seral\thanks{Departamento de Matem\'aticas, Universidad de Zaragoza, Pedro Cerbuna, 12, 50009 Zaragoza, Spain; \texttt{paz@unizar.es}; ORCID 0000-0003-4809-1784}\and V. P\'erez-Calabuig\addtocounter{footnote}{-2}\footnotemark}
\date{}
\begin{document}
\maketitle
\begin{abstract}
A complete description of all possible multiplicative groups of finite skew left braces whose additive group has trivial centre is shown. As a consequence, some earlier results of Tsang can be improved and an answer to an open question set by Tsang at Ischia Group Theory 2024 Conference is provided. 
  

\end{abstract}
\section{Introduction}

An important algebraic structure that plays a 
key role in the combinatorial theory of Yang-Baxter equation is the skew brace structure. Skew left braces, introduced in~\cite{GuarnieriVendramin17},  can be regarded as extensions of Jacobson radical rings and show connections with several areas of mathematics such as triply factorised groups and Hopf-Galois structures (see~\cite{BallesterEsteban22, CarantiStefanello21, Childs18})

Skew left braces classify solutions of the Yang-Baxter equation (see~\cite{EtingofSchedlerSoloviev99, GuarnieriVendramin17}). This justifies the search for constructions of skew braces and classification results. 

Recall that a \emph{skew left brace} is a set endowed with two group structures $(B,+)$ and $(B,\cdot)$ which  are linked by the distributive-like law ${a(b+c)}=ab-a+ac$ for $a$, $b$, $c\in B$. 


In the sequel, the word \emph{brace} refers to a skew left brace.  

Given a brace $B$, there is an action of the multiplicative group on the additive group by means of the so-called \emph{lambda map}:
\[ \lambda \colon a \in (B,\cdot) \longmapsto \lambda_a \in \Aut(B,+)), \quad \lambda_a(b)=-a+ab, \ \text{for all $a,b\in B$.}\]
Braces can be described in terms of  regular subgroups of the holomorph of the additive group. Recall that the holomorph of a group $G$ is the semidirect product $\Hol(G) = [G]\Aut(G)$. Let $B$ be a brace and set $K = (B,+)$. Then $H=\{ (a,\lambda_a) \mid a\in B\}$ is a regular subgroup of the holomorph $\Hol(K)$ isomorphic to $(B, {\cdot})$ (see~\cite[Theorem 4.2]{GuarnieriVendramin17}). If we consider the subgroup $S = KH \leq \Hol(K)$, 
then
\[S=KH=KE=HE,\]
where $E = \{(0,\lambda_b)\mid b \in B\}$ and $\Cent_{E}(K)= K\cap E =H\cap E =1$. We call $\mathsf{S}(B)=(S, K, H, E)$ the \emph{small trifactorised group} associated with~$B$.



In \cite{Tsang23-blms}, Tsang showed it is possible to construct finite braces by just looking at the automorphism group of the additive group instead of looking at the whole holomorph. This is a significant improvement both from an algebraic and computational approach. 

\begin{theorem}[{see \cite[Corollary~2.2]{Tsang23-blms}}]\label{th-Tsang-2.2}
  If the finite group $G$ is the multiplicative group of a brace with additive group~$K$, then there exist two subgroups $X$ and $Y$ of $\Aut(K)$ that are quotients of $G$ satisfying
  \[XY=X{\Inn(K)}=Y{\Inn(K)}.\]
\end{theorem}

She looks for a sort of converse of the above theorem in the case of finite braces with an additive group of trivial centre, and proved the following: 

\begin{theorem}[{see \cite[Proposition~2.7]{Tsang23-blms}}]\label{th-prop-2.7}
  Suppose that the centre of a finite group $(K, {+})$ is trivial and let $P$ be a subgroup of $\Aut(K)$ containing $\Inn(K)$. If $P=XY$ is factorised by two subgroups $X$ and $Y$ such that $X\cap Y=1$, $X{\Inn(K)}=Y{\Inn(K)}=P$ and $X$ splits over $X\cap \Inn(K)$, then there exists a brace $B$ whose additive group is isomorphic to $(K, {+})$ and whose multiplicative group is isomorphic to a semidirect product $[X\cap \Inn(K)]Y$ for a suitable choice of the action $\alpha\colon Y\longrightarrow \Aut(X\cap \Inn(K))$.  
\end{theorem}

The above two theorems are the key to prove the main results of \cite{Tsang23-blms,Tsang24-agta}.

In \cite{Tsang-Ischia24}, Tsang posed the following question:

\begin{question}\label{quest-Tsang-Ischia}
Is it possible to extend Theorem~\ref{th-prop-2.7} by dropping the assumption that $X$ splits over $X\cap\Inn(K)$?
\end{question}

The aim of this paper is to give a complete characterisation of the multiplicative groups of a brace with additive group of trivial centre. As a consequence, we present an improved version of Theorem~\ref{th-prop-2.7} (on which the main result of~\cite{Tsang23-blms} heavily depends), and we give an affirmative answer to Question~\ref{quest-Tsang-Ischia}.

\begin{athm}\label{th-trivz}
  Let $K$ be a finite group with trivial centre. For every brace $B$ with additive group $K = (B,+)$ and multiplicative group $C = (B,\cdot)$, there exist subgroups $X$ and $Y$ of $Aut(K)$ satisfying the following properties: 
\begin{enumerate}[(a)]
    \item $XY=X{\Inn(K)}=Y{\Inn(K)}$,
    \item $\Inn(K)$ has two subgroups $N$ and $M$ such that $N\trianglelefteq X$ and $M\trianglelefteq Y$ and there exists an isomorphism $\gamma\colon Y/M\longrightarrow X/N$ such that whenever $\gamma(yM)=xN$, it follows that $xy^{-1}\in \Inn(K)$. Furthermore, if $z\in\Inn(K)$, there exist $x\in X$, $y\in Y$ such that $xN=\gamma(yM)$ and $xy^{-1}=z$,
    \item $\lvert K\rvert=\lvert X\rvert \lvert M\rvert=\lvert Y\rvert \lvert N\rvert$.
\end{enumerate}
In this case, 
\begin{enumerate}[(a)]
\setcounter{enumi}{3}
\item $C$ has two normal subgroups $T$ and $V$ with $T\cap V=1$, $X\cong C/T$ and $Y\cong C/V$.
\end{enumerate}

Conversely, for every pair $X$, $Y$ of subgroups of $\Aut(K)$ satisfying conditions (a)--(c), there exists a brace $B$ with $K = (B,+)$ and $C= (B,\cdot)$ satisfying (d).
\end{athm}

\begin{corollary}
\label{cor:drop-split}
Let $K$ be a finite group with trivial centre. Suppose that there exist subgroups $X$, $Y$ of $\Aut(K)$ such that $X \cap Y = 1$ and $XY=X{\Inn(K)}=Y{\Inn(K)}$. Then there exists 
a brace with additive group $K$ and a multiplicative group which is  isomorphic to a subdirect product of $X$ and~$Y$. 
\end{corollary}

\begin{proof}
Assume that $X\cap Y=1$. Consider $N=X\cap {\Inn(K)}$, $M=Y\cap {\Inn(K)}$. Then $\lvert X\rvert\lvert M\rvert=\lvert K\rvert$
 as $\lvert X\rvert \lvert Y\rvert=\lvert {\Inn(K)}\rvert\lvert Y\rvert/\lvert Y\cap {\Inn(K)}\rvert$. Analogously, $\lvert Y\rvert\lvert N\rvert=\lvert K\rvert$. Moreover, since
\[Y/M\cong Y{\Inn(K)}/{\Inn(K)}=X{\Inn(K)}/{\Inn(K)}\cong X/N,\]
we have an isomorphism $\gamma\colon Y/M\longrightarrow X/N$ given by $\gamma(bM)=aN$, where $b \in Y$, $a \in X$ such that $ab^{-1} \in \Inn(K)$. Observe that $a$ and $b$ are unique as $X \cap Y = 1$. Then, the groups $X$ and $Y$ satisfy Statements~(a)--(c) of Theorem~\ref{th-trivz}, and therefore, there exists a brace whose additive group is $K$ and a multiplicative group isomorphic to a subdirect product of $X$ and~$Y$.  
\end{proof}


Corollary~\ref{cor:drop-split} also allows to give a considerably shorter proof of the main results of \cite{Tsang23-blms,Tsang24-agta} about the almost simple groups $K$ that can appear as additive groups of braces with soluble multiplicative group. By Corollary~\ref{cor:drop-split}, it is enough to find two subgroups $X$ and $Y$ of $\Aut(K)$ such that $X \cap Y = 1$ and $XY=X{\Inn(K)}=Y{\Inn(K)}$. Therefore, Codes~2, 3, and~4 in the proof of~\cite[Theorem 1.3]{Tsang23-blms} can be avoid, as well as checking in every case that the subgroup $X$ splits over $X\cap {\Inn(K)}$. 

    

In Section~\ref{sec:worked-ex} we present a worked example of a construction of a brace with additive group $K = \PSL_2(25)$ by means of subgroups $X$ and  $Y$ of $\Aut(K)$ satisfying all conditions of Theorem~\ref{th-trivz} but $X \cap Y \neq 1$.

\section{Proof of Theorem~\ref{th-trivz}}\label{sec-Tsang}

\begin{proof}[Proof of Theorem~\ref{th-trivz}]
Suppose that $B$ is a brace with additive group $K$ and lambda map~$\lambda$. Let $H=\{(b, \lambda_b)\!\mid b\in B\}$ the regular subgroup of $\Hol(K)$ appearing in the small trifactorised group $\mathsf{S}(B)=(S, K, H, E)$ associated with~$B$. Recall that  $E= \{(0, \lambda_b)\!\mid b \in B\} \leq \Hol(K)$, and $S = KH = KE = HE$ with $K\cap E=H\cap E= 1$. 
  
Observe that $S$ acts on $K$ by means of the homomorphism $\pi \colon (b, \omega)\in S \mapsto \omega \in \Aut(K)$. On the other hand, $S$ also acts on $K$ by conjugation. In fact, this action naturally induces a homomorphism $\alpha \colon S \rightarrow \Aut(K)$. In particular, for every $b\in B$,  $(0,\lambda_b)(k,1)(0,\lambda_b)^{-1}=(\lambda_b(k), 1)$, that is, $\alpha(0,\lambda_b) = \lambda_b = \pi(0,\lambda_b)$. Thus, $\alpha(E) = \pi(E) = \pi(H)$.

  \begin{figure}[h]
    \centering
    \[\xymatrix{&H\ar@{-}[d]\\
      (X)&(H\cap K){\Cent_H(K)}\ar@{-}[dl]\ar@{-}[dr]&(Y)\\
      {\Cent_H(K)}\ar@{-}[dr]&&H\cap K\ar@{-}[dl]\\
      &1}\]
    \caption{Structure of the multiplicative group in Theorem~\ref{th-trivz}}
    \label{fig-subdirect}
  \end{figure}
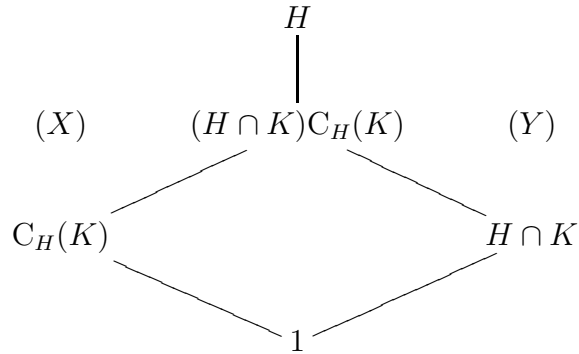

The restrictions of $\pi$ and $\alpha$ to $H$ induce two actions of $H$ on~$K$, with respective kernels $\Ker{\pi}|_H=K\cap H\trianglelefteq H$ and $\Ker{\alpha|_H}=\Cent_H(K)\trianglelefteq H$. Moreover, it holds that
\begin{align*}
& \Ker{\pi}|_H\cap \Ker{\alpha|_H}= K\cap H\cap \Cent_H(K)= K\cap H\cap \Cent_S(K)= \\
& = H\cap \Cent_{K}(K)=H\cap \centre(K)=1 \quad \text{(see Figure~\ref{fig-subdirect}).}
\end{align*}

  Let $X:=\alpha(H)$ and $Y := \pi(H) = \alpha(E) = \{\lambda_b\mid b\in B\}$ such that $X \cong H/{\Cent_H(K)}$ and $Y \cong H/(K\cap H)$. Since $\alpha(K)=\Inn(K)$, we have that
  \begin{align*}
    \alpha(G)&=\alpha(HE)=\alpha(KH)=\alpha(KE)\\
             &=\alpha(H)\alpha(E)=\alpha(K)\alpha(H)=\alpha(K)\alpha(E)\\
             &=XY=(\Inn(K))X=(\Inn(K))Y.
  \end{align*}

Take $R:=(H\cap K){\Cent_H(K)} \unlhd H$. Then, $N:=\alpha(R)\trianglelefteq \alpha(H)=X$ and $ M:= \pi(R)\trianglelefteq \pi(H) = Y$. It follows that $N =\alpha(H\cap K)\le \alpha(K)=\Inn(K)$. On the other hand, $M = \pi(\Cent_H(K))$ and if $(b,\lambda_b)\in{\Cent_H(K)}$, then for every $k\in K$,
\[
    (b, \lambda_b)(k,1)(b,\lambda_b)^{-1}=(b+\lambda_b(k)-b,1)=(k,1),
  \]
  that is, $\lambda_b$ coincides with the inner automorphism of $K$ induced by~$-b$. Thus, $M \le \Inn(K)$. Moreover, we see that
  \begin{align*} Y/M &\cong (H/\Ker{\pi}|_H)/(R/\Ker{\pi}|_H)\cong H/R\\
    &\cong(H/\Ker\alpha|_H)/(R/\Ker\alpha|_H)\cong X/N;
  \end{align*}
  here the isomorphism $\gamma\colon Y/M\longrightarrow X/N$ is given by $\gamma(\lambda_b M)=\alpha_b\lambda_b N$, where $\alpha_b$ is the inner automorphism of $K$ induced by~$b$. Given $a\in \gamma(\lambda_bM)$, we have that $a\lambda_b^{-1}\in\alpha_bN\subseteq \Inn(K)$. Furthermore, given $x\in \Inn(K)$, we have that $x=\alpha_b$ for some $b\in B$ and so $\gamma(\lambda_b M)=\alpha_b\lambda_b N=x\lambda_b N$ with $(\alpha_b\lambda_b)\lambda_b^{-1}=x$.

Since $\Ker{\pi|}_H\cap {\Ker\alpha|_H}=(H\cap K)\cap \Cent_H(K)=1$, we have that $\lvert R\rvert={\lvert H\cap K\rvert}\lvert {\Cent_H(K)}\rvert$ and $\lvert N\rvert=\lvert R/(H\cap K)\rvert=\lvert \Cent_H(K)\rvert$, $\lvert M\rvert=\lvert R/{\Cent_H(K)}\rvert=\lvert H\cap K\rvert$. As $\lvert X\rvert=\lvert K\rvert/\lvert H\cap K\rvert$ and $\lvert Y\rvert=\lvert K\rvert/\lvert{\Cent_H(K)}\rvert$, the claim about the order follows.

\bigskip

  
Now, suppose that $\Aut(K)$ possesses subgroups $X$ and $Y$ satisfying conditions (a)--(c). Let
  \[W=\{(x,y)\mid x\in X,\, y\in Y,\, \gamma(yM)=xN\}\]
be a subdirect product of $X$ and $Y$ with amalgamated factor group $Y/M\cong X/N$ (see \cite[Chapter~A, Definition~19.2]{DoerkHawkes92}). By \cite[Chapter~A, Proposition~19.1]{DoerkHawkes92}, and the hypothesis, we have that $\lvert W\rvert=\lvert K\rvert$. Since $\centre(K)$ is trivial, the map $\zeta\colon K\longrightarrow \Inn(K)$, where $\zeta(k)$ is the inner automorphism of $K$ induced by~$k$, is an isomorphism. By hypothesis, the map $W\longrightarrow \Inn(K)$ given by $(x,y)\longmapsto xy^{-1}$ is surjective. Since $\lvert W\rvert=\lvert\Inn(K)\rvert=\lvert K\rvert$, it is a bijection. We can consider $H=\{(b,y)\mid (x,y)\in W,\, \zeta(b)=xy^{-1}\}\subseteq\Hol(K)$. Given $(b, y)$, $(b_1, y_1)\in H$, we have that $(b, y)(b_1,y_1)=(b+y(b_1),yy_1)$, $\zeta(b)=xy^{-1}$, and $\zeta(b_1)=x_1y_1^{-1}$ with $(x,y)$, $(x_1, y_1)\in B$. Then
\[\zeta(b+y(b_1))=\zeta(b)\zeta(y(b_1))=\zeta(b)y\zeta(b_1)y^{-1}=xy^{-1}yx_1y_1^{-1}y^{-1}=(xx_1)(yy_1)^{-1}\]
with $(xx_1, yy_1)=(x,y)(x_1,y_1)\in W$. Furthermore, if $(b, y)\in H$, with $\zeta(b)=xy^{-1}$, we have that $(b, y)^{-1}=(y^{-1}(-b), y^{-1})$ and
\[\zeta(y^{-1}(-b))=y^{-1}\zeta(-b)y=y^{-1}\zeta(b)^{-1}y=y^{-1}yx^{-1}y=x^{-1}{(y^{-1})}^{-1}\]
with $(x^{-1}, y^{-1})=(x,y)^{-1}\in W$. We conclude that $H$ is a subgroup of $\Hol(K)$. As the projection onto its first component is surjective, it turns out that it $H$ is a regular subgroup of $\Hol(G)$ by \cite[Proposition~2.5]{BallesterEstebanPerezC24-braces64-55} and so it is isomorphic to the multiplicative group of a brace with additive group~$K$ (see \cite[Theorem 4.2]{GuarnieriVendramin17}).

We finish the proof by showing that the map $\gamma\colon H\longmapsto W$ given by $(b, y)\longmapsto (\zeta(b)y, y)$, where $\zeta(b)=xy^{-1}$ and $(x, y)\in W$, is an isomorphism. Indeed, if $\zeta(b)=xy^{-1}$, $\zeta(b_1)=x_1y_1^{-1}$, where $(x,y)$, $(x_1,y_1)\in W$, we have that
\begin{align*}
    \gamma(b,y)\gamma(b_1,y_1)&=(\zeta(b)y,y)(\zeta(b_1)y_1,y_1)=(x,y)(x_1,y_1)=(xx_1,yy_1),\\
    \gamma((b,y)(b_1,y_1))&=\gamma(b+y(b_1),yy_1)=(\zeta(b+y(b_1))yy_1,yy_1)\\
                              &=(\zeta(b)y\zeta(b_1)y^{-1}yy_1,yy_1)=(xy^{-1}yx_1y_1^{-1}y_1,yy_1)\\
                              &=(xx_1,yy_1).                               
  \end{align*}
We conclude that $\gamma$ is a group homomorphism. Assume that $\gamma(b, y)=(\zeta(b)y,y)=(1,1)$, with $\zeta(b)=xy^{-1}$ and $(x,y)\in H$, then $y=1$ and so $\zeta(b)=x=1$, what implies that $b=0$. Consequently, $\gamma$ is injective. As $W$ and $H$  are finite and have the same order, we obtain that $\gamma$ is an isomorphism.
\end{proof}

\section{A worked example}
\label{sec:worked-ex}

In general, we do not have that $X\cap Y=1$.  Let us consider $K=\PSL_2(25)$. Its automorphism group $A=\Aut(K)$ is generated by $\Inn(K)$, the diagonal automorphism $d$ induced by the conjugation by the matrix
  \[\mathsf{D}=
    \begin{bmatrix}
      \zeta&0\\
      0&1
    \end{bmatrix}\in\GL_2(25),\]
  where $\zeta$ is a primitive $24$th-root of unity of $\GF(25)$, and the field automorphism $f$. The group $A$ possesses a subgroup $X$ generated by the inner automorphisms $c_1$,  $c_2$, and $c_3$ induced by the matrices
  \[\mathsf{C}_1=
    \begin{bmatrix}
      \zeta^{11}&0\\
      \zeta^{14}&\zeta^{13}
    \end{bmatrix},\qquad\mathsf{C}_2=
    \begin{bmatrix}
      1&0\\
      \zeta&1
    \end{bmatrix},\qquad \mathsf{C}_3=
    \begin{bmatrix}
      1&0\\
      1&1
    \end{bmatrix}
  \]
  respectively, and, if $c_0$ is the inner automorphism of $K$ induced by
  \[\mathsf{C}_0=
    \begin{bmatrix}
      \zeta^{7}&0\\
      \zeta^{17}&\zeta^{17}
    \end{bmatrix},
  \]
  the element $c_0fd$. We have that $c_1$ has order $12$, $\langle c_2, c_3\rangle$ is an elementary abelian group of order~$25$, and $(c_0fd)^2$ is induced by the matrix $\mathsf{R}^2\mathsf{C}_1^3$ of  order~$8$, with
  \[\mathsf{R}=
    \begin{bmatrix}
      \zeta&0\\
      0&\zeta
    \end{bmatrix}\in\centre(\GL_2(25)).
  \]
  The group $\langle c_1, c_2, c_3,  c_0fd\rangle$ has order~$600$.

  It also possesses a subgroup $Y$ generated by the inner automorphisms $u_1$ and $u_2$ induced by the conjugation by
  \[\mathsf{U_1}=\begin{bmatrix}
      3&\zeta^{21}\\
      \zeta^{14}&\zeta^{22}
    \end{bmatrix},\qquad
    \mathsf{U_2}=
    \begin{bmatrix}
      3&2\\
      0&2
    \end{bmatrix},
  \]
  respectively, and $u_0fd$, where $u_0$ is the inner automorphism induced by the conjugation by
  \[\mathsf{U_0}=
    \begin{bmatrix}
      \zeta&4\\
      \zeta^7&\zeta^8
    \end{bmatrix}.
  \]
  Then $(u_0fd)^2$ is induced by $\mathsf{R}^{12}\mathsf{U}_2$, that is, $(u_0fd)^2=u_2$, $(u_0fd)^2$ has order~$2$, $u_1$ has order~$13$, and $\langle u_1, (u_0fd)^2\rangle$ has order~$52$. By \cite{Atlas85}, $X$ and $Y$ are maximal subgroups of the almost simple group $\Inn(K)\langle fd\rangle$.

  We note that $\Inn(K)X=\Inn(K)Y=\Inn(K)\langle fd\rangle$. Moreover, $\lvert {X\cap Y}\rvert$ divides $\gcd (\lvert X\rvert, \lvert Y\rvert)=4$ and, since $X$ and $Y$ have cyclic Sylow $2$-subgroups, $X\cap Y$ must be cyclic and so $\lvert X\cap Y\rvert \le 4$. But if $\lvert X\cap Y\rvert = 4$, then $X\cap Y$ is contained in $X\cap \Inn(K)$, but it is not contained in $Y\cap \Inn(K)$. This shows that $\lvert X\cap Y\rvert \le 2$. 
  As $XY\subseteq \Inn(K)\langle fd\rangle$,
  \[15\,600=\lvert \Inn(K)\langle fd\rangle\rvert\ge \lvert XY\rvert=\frac{\lvert X\rvert \lvert Y\rvert}{\lvert X\cap Y\rvert}=15\,600\cdot \frac{2}{\lvert X\cap Y\rvert}\ge 15\,600,\]
  and so $XY=\Inn(K)\langle fd\rangle$ and $\lvert X\cap Y\rvert=2$. In fact, the inner automorphism $t$ induced by the matrix
  \[\mathsf{T}=
    \begin{bmatrix}
      2&\zeta^8\\
      0&3
    \end{bmatrix}
  \]
  belongs to $X\cap Y$, as $\mathsf{T}=\mathsf{C}_1^{6}\mathsf{C}_2\mathsf{C}_3$ induces by conjugation a non-trivial automorphism of~$X$ and $\mathsf{T}=\mathsf{U}_1^3\mathsf{U}_2^3$
  induces a non-trivial automorphism by conjugation on $\SL_2(25)$ that belongs to~$Y$.

  Let $N=\langle c_2, c_3,c_1^2\rangle\trianglelefteq X$, $M=\langle u_1\rangle\trianglelefteq Y$. Then $\lvert N\rvert=150$, $\lvert M\rvert=13$, $N\le X\cap \Inn(K)$, $M\le Y\cap \Inn(K)$,  $Y/M\cong X/N\cong C_4$, and $\lvert K\rvert=\lvert X\rvert \lvert M\rvert = \lvert Y\rvert \lvert N\rvert$. The isomorphism between $Y/M$ and $X/N$ is given by $\gamma((u_0fd)^rM)=(c_0fd)^rN$ for $0\le r\le 4$, and it is clear that
  \begin{gather*}
    (c_0fd)(u_0fd)^{-1}=c_0u_0^{-1}\in \Inn(K),\\
    (c_0fd)^2(u_0fd)^{-2}\in \Inn(K),\\ (c_0fd)^3(u_0fd)^{-3}=(c_0fd)^2(c_0u_0^{-1})(u_0fd)^{-2}\in\Inn(K). 
  \end{gather*}

  Let $z\in XY\cap \Inn(K)$. Then there exist $x\in X$, $y\in Y$ with $z=xy^{-1}=(xt)(yt)^{-1}$, where $\langle t\rangle=X\cap Y$. We observe that $t\in N$, as $\mathsf{T}=\mathsf{C}_1^{6}\mathsf{C}_2\mathsf{C}_3$, but $t\notin M$ by order considerations. Given $x\in X$, $y\in Y$, there exist $r$, $s\in\{0,1,2,3\}$ such that $xN=(c_0fd)^rN$ and $yM=(u_0fd)^sM$. We also observe that $x\in \Inn(K)$ if, and only if, $y\in \Inn(K)$. To prove that we can choose $x\in X$, $y\in Y$ such that $z=xy^{-1}$ and  $\gamma(yM)=xN$, it is enough to prove that for such a choice we have that $z=xy^{-1}$ and $r=s$. Note that if $x\in N$, then $r=0$; if $x\in \Inn(K)\setminus N$, then $r=2$; and if $x\notin\Inn(K)$, then $r\in\{1,3\}$. Analogously, if $y\in M$, then $s=0$; if $y\in\Inn(K)\setminus M$, then $s=2$; and if $y\notin \Inn(K)$, then $s\in \{1,3\}$. We also have that $tM=(u_0fd)^2M$ and that $tN=N$, as $t\in \Inn(K)$, $t\in N$, but $t\notin M$. If $x\in N$ and $y\in M$, we can choose $r=s=0$ and $\gamma(yM)=xN$. Suppose that $x\in N$ and $y\notin M$. Then $y\in\Inn(K)$ and so, $xN=N$ and $yM=(u_0fd)^2M$. Consequently, $xtN=N$, $ytN=N$, and $\gamma(ytN)=xtN$. Suppose that $x\notin N$ and $y\in M$. We have that $x\in \Inn(K)$ and so, $xN=(c_0fd)^2N$ and $yM=M$. It follows that $xtN=(c_0fd)^2N$ and $ytM=(u_0fd)^2M$, that is, $\gamma(ytM)=xtN$. Suppose that $x$, $y\in \Inn(K)$, $x\notin N$, and $y\notin M$. Then $xN=(c_0fd)^2N$, $yM=(u_0fd)^2M$, and $\gamma(yM)=xN$. Finally, suppose that $x$ and  $y\notin M$. Then $xN=(c_0fd)^rN$ and $yM=(u_0fd)^sM$, with $r$, $s\in\{1,3\}$. If $r=s$, then $\gamma(yM)=xN$. If $r\ne s$, then $xtN=(c_0fd)^rN$ and $ytM=(u_0fd)^{s+2}M$, with $r\equiv s+2\pmod{4}$. Thus $\gamma(ytM)=xtN$.

  It follows that $X$, $Y$ satisfy all conditions of Theorem~\ref{th-trivz}.
  We can also check with \textsf{GAP} \cite{GAP4-12-2} all this information about these subgroups.

 \providecommand{\iflanguage}[3]{#3}\relax \providecommand{\mathfrak}{\mathcal}
  \providecommand{\acceptedin}{\iflanguage{spanish}{Aceptado en}{Accepted in}}
  \providecommand{\accepted}{\iflanguage{spanish}{Aceptado}{Accepted}}
  \providecommand{\talkat}{\iflanguage{spanish}{Charla en}{Talk at}}
  \providecommand{\plenarytalkat}{\iflanguage{spanish}{Charla plenaria
  en}{Plenary talk at}} \providecommand{\JIF}{JIF}
  \providecommand{\PhDcourseat}{\iflanguage{spanish}{Curso de doctorado en}{PhD
  course at}} \providecommand{\transrus}{\iflanguage{spanish}{Traducci{\'o}n
  del art{\'\i}culo original en ruso en}{Translation of the original paper in
  Russian from}}
  \providecommand{\supervisedby}{\iflanguage{spanish}{Direcci{\'o}n:}{Supervised
  by}} \providecommand{\visited}{\iflanguage{spanish}{visitada}{visited}}
  \providecommand{\inpress}{\iflanguage{spanish}{en
  prensa}{\iflanguage{catalan}{en premsa}{in press}}}
  \providecommand{\encurs}{\iflanguage{spanish}{en
  curso}{\iflanguage{catalan}{en curs}{in course}}}
  \providecommand{\aand}{\iflanguage{spanish}{y}{and}}
  \providecommand{\invitedtalkat}{\iflanguage{spanish}{Conferencia invitada
  en}{Invited talk at}} \providecommand{\posterat}{\iflanguage{spanish}{Póster
  en}{Poster at}}
  \providecommand{\invitedtalk}{\iflanguage{spanish}{Conferencia
  invitada}{Invited talk}}
  \providecommand{\oral}{\iflanguage{spanish}{Comunicaci{\'o}n oral}{Oral
  communication}}
  \providecommand{\Poland}{\iflanguage{spanish}{Polonia}{Poland}}
  \providecommand{\Germany}{\iflanguage{spanish}{Alemania}{Germany}}
  \providecommand{\Italy}{\iflanguage{spanish}{Italia}{Italy}}
  \providecommand{\Ireland}{\iflanguage{spanish}{Irlanda}{Ireland}}
  \providecommand{\Hungary}{\iflanguage{spanish}{Hungr{\'\i}a}{Hungary}}
  \providecommand{\Portugal}{\iflanguage{spanish}{Portugal}{Portugal}}
  \providecommand{\Cairo}{\iflanguage{spanish}{El Cairo, Egipto}{Cairo, Egypt}}
  \providecommand{\Bath}{\iflanguage{spanish}{Bath, Inglaterra, Reino
  Unido}{Bath, England, United Kingdom}}
  \providecommand{\presentada}[2]{\iflanguage{spanish}{presentada por
  #1}{presented by #2}}
  \providecommand{\Ponenciaplenaria}{\iflanguage{spanish}{Ponencia
  plenaria}{Plenary talk}}
  \providecommand{\Naples}{\iflanguage{spanish}{N{\'a}poles}{Naples}}
  \def\cprime{$'$} \providecommand{\germ}{\mathfrak}
  \providecommand{\url}{\texttt}

\end{document}